\documentclass[12pt,a4paper,oneside]{amsart}

\usepackage{geometry}
\newtheorem{theorem}{Theorem}[section]
\newtheorem{lemma}[theorem]{Lemma}

\theoremstyle{definition}
\newtheorem{definition}[theorem]{Definition}

\theoremstyle{remark}
\newtheorem{remark}[theorem]{Remark}

\theoremstyle{corollary}
\newtheorem{corollary}[theorem]{Corollary}

\theoremstyle{question}
\newtheorem{question}[theorem]{Question}

\theoremstyle{conjecture}
\newtheorem{conjecture}[theorem]{Conjecture}

\numberwithin{equation}{section}

\usepackage{hyperref}

\usepackage[all]{xy}
\usepackage{amssymb}
\usepackage{tikz}
\usepackage{tikz-cd}
\usetikzlibrary{
	calc,
	decorations.pathmorphing,
	matrix,arrows,
	positioning,
	shapes.geometric
}
\usepackage{pgfplots}
\pgfplotsset{compat=newest}

\begin{document}

\title[Topological properties on isochronous centers]{ Topological properties on isochronous centers of  polynomial Hamiltonian differential systems}


\author[G. Dong]{Guangfeng Dong}
\address{Department of Mathematics, Jinan University, Guangzhou 510632, China}
\curraddr{}
\email{donggf@jnu.edu.cn}



\thanks{}






\begin{abstract}
In this paper, we study the topological properties of complex polynomial Hamiltonian differential systems of degree $n$ having an isochronous center.
Firstly, we prove that if  the critical level curve 
possessing  an isochronous center contains only a single singular point, and the period $1$-form does not have poles with zero residue at infinity on level curves sufficiently close to the critical curve,
then the vanishing cycle associated to this center is trivial in the 1-dimensional homology group of the projective closure of a generic level curve.
Our result provides a positive answer to a question asked by L. Gavrilov under 
relatively simple conditions and can be applied to achieve an equivalent description of  the Jacobian conjecture on $\mathbb{C}^2$.
Secondly, we obtain a very simple but useful
 necessary condition for isochronicity of Hamiltonian systems, which is that
the $(n+1)$-degree part of the Hamiltonian function must have a factor with multiplicity no less than $(n+1)/2$.
Thirdly, we show a relation between Gavrilov's question and the conjecture proposed by X. Jarque and J. Villadelprat on the non-isochronicity of real Hamiltonian systems of even degree $n$.
\end{abstract}

\maketitle

\section{Introduction and main results}
\label{intro}

Consider the following complex polynomial Hamiltonian differential systems of degree $n$
\begin{eqnarray}\label{H-C}
   \left(
   \begin{array}{c}
   	\frac{d x}{d t} \\
   	\frac{d y}{d t} \\
   \end{array}
   \right)
   =
  \left(
   \begin{array}{r}
   	-\frac{\partial H}{\partial y} \\
   	\frac{\partial H}{\partial x} \\
   \end{array}
   \right),  \  \ (x,y)\in \mathbb{C}^2,\  t\in \mathbb{C},
\end{eqnarray}
where the   Hamiltonian function
$H(x,y)$ is a polynomial of degree $n+1$ in $ \mathbb{C}[x,y]$. 
Assuming the origin $O$ is a center of Morse type, without loss of generality,
$H(x,y)$ can be written as  $H(x,y)=(x^2 +y^2)/2+ h.o.t.$.
For a generic  level curve $L_h$  defined by the algebraic equation $H(x,y)=h$ 
where $h\in \mathbb{C}$ is sufficiently close to $0$,
one can associate a  \emph{vanishing cycle} $\gamma_h$ to the critical value $h=0$,
which is a 1-dimensional cycle vanishing at $h=0$ in the 1-dimensional homology group
$\mathcal{H}_{1}(L_{h},\mathbb{Z})$ and
can be characterized by the following purely
topological property: 
modulo orientation and the free homotopy deformation on $L_{h}$, 
as $h\rightarrow 0$,
the cycle $\gamma_h$ can be represented by a
continuous family of loops on $L_{h}$ of length that tends to zero.
This description explains the terminology(see, e.g., \cite{Ilya-Yako}).
Respectively $T(h)=\oint_{\gamma_h} dt$ is called a period function of system (\ref{H-C}).
If $T(h)$ is a nonzero constant independent of $h$ for $h\neq 0$,
then the origin is called an \textit{isochronous center}.
This definition coincides with the classical isochronous center when $(x,y)\in \mathbb{R}^2$ and $t\in \mathbb{R}$.

One of the most important problems on isochronous centers is 
to describe the role of the vanishing cycle $\gamma_h$ in  the 1-dimensional homology group 
of the compact Riemann surface of $L_h$.
It is still an open problem until now. 
In \cite{Gavri}, L. Gavrilov has asked the following question for systems (\ref{H-C}) with only isolated singularities:
\begin{question}[Gavrilov's question]\label{topo-iso-q}
Is it true that if a Morse singular point is isochronous, then the associated vanishing cycle
 represents a zero homology cycle on the Riemann surface of the level curve $L_h$?
\end{question}


In  general cases, the above question has a negative answer.
Example 3.23 in reference \cite{C-M-V} provides a 
system with
$$H(x,y)=x^2 (x^2+2)(x^2+4)+2x^2 (x^2+1)(x^2+2)(x^2+3)y+(x^2+1)^4(x^2+2)y^2,$$ 
which has an isochronous center at the origin,
but the corresponding vanishing cycle is not homologous to zero  on the Riemann surface of $L_h$.
In this counterexample,
it is not difficult to see that the critical level curve $L_0$ contains at least three  different singularities on $\mathbb{C}^2$.

What conditions can give a positive answer to Gavrilov's question? 
This  is also an important and meaningful question,
especially it is closely related with  the famous Jacobian conjecture on $\mathbb{C}^2$,
which asserts that
the following  polynomial map with a constant Jacobian determinant
\begin{eqnarray}\label{J-Homo}
	\begin{array}{lrll}
		\Phi_P: &\mathbb{C}^2 &\longrightarrow &\mathbb{C}^2 \\
		& (x,y)&\longmapsto &	(f(x,y), g(x,y))
	\end{array}
\end{eqnarray}
is a global homeomorphism, where $f=x+h.o.t.$ and $ g=y+ h.o.t. $  are polynomials in $ \mathbb{C}[x,y]$.
At present it has been proved only when the degrees
of $f$ and $g$ are not too large.
Obviously the map $\Phi_P $  induces a Hamiltonian system 
\begin{eqnarray}\label{J-H-C}
	\left(
	\begin{array}{c}
		\frac{d x}{d t} \\
		\frac{d y}{d t}\\
	\end{array}
	\right)=
	\left(
	\begin{array}[2]{r}
		-f\frac{\partial f}{\partial y} -g\frac{\partial g}{\partial y}\\
		f\frac{\partial f}{\partial x} +g\frac{\partial g}{\partial x}
	\end{array}
	\right)=
	\left(
	\begin{array}[2]{r}
		-\frac{\partial H}{\partial y}\\
		\frac{\partial H}{\partial x}
	\end{array}
	\right)
\end{eqnarray}
having  an isochronous center of Morse type at the origin with the Hamiltonian function
$H(x,y)=(f^2 +g^2)/2$.

Also in  \cite{Gavri}, Proposition 6.1   says that  
if  the vanishing cycle associated to the origin for system (\ref{J-H-C}) represents a zero homology cycle on the Riemann surface of a generic level curve, then the map $\Phi_P $ is injective, which suffices to guarantee the Jacobian conjecture  is true.
In addition, he has also proved that(Theorem 4.1 of  \cite{Gavri}) Question \ref{topo-iso-q} has  a positive answer under the conditions that the critical level curve $L_0$ contains only a single singular point which is isochronous
and $H(x,y)$ is a `good' polynomial having only isolated and simple singularities,
where the definition of a good polynomial depends on  the Milnor numbers of the  complex projective closure $\overline{L_h}$ of $L_h$ at infinity.

This paper is devoted to look for other conditions 
 to give a positive answer to Question \ref{topo-iso-q}. 
 Denote by $$\omega=dt = -\frac{d x}{H_{y}  },\  \ H_{y}=\frac{\partial H}{\partial y}, $$ the period $1$-form of system (\ref{H-C}).
We have the following main theorem.
\begin{theorem}\label{h-0-c}
For system (\ref{H-C}), if the critical level curve $L_0$ contains a single singularity which is an isochronous center of Morse type, and the period $1$-form $\omega$ does not have poles with zero residue at infinity for any $h$  sufficiently close to $0$, then the associated vanishing cycle $\gamma_h$ is trivial in $\mathcal{H}_1(\overline{L_h},\mathbb{Z})$.
\end{theorem}

Applying the above theorem to system (\ref{J-H-C}), 
one can achieve an equivalent description of the Jacobian conjecture. 

\begin{corollary}\label{J-C}
 The  polynomial map $\Phi_P$ with constant Jacobian determinant  is a global homeomorphism,
if and only if  two algebraic curves $f=0$ and $g=0$ intersect only at a single point on $\mathbb{C}^2$. 
\end{corollary}

To prove Theorem \ref{h-0-c}, we will carefully study 
some real systems
induced by complex system (\ref{H-C}) 
and the corresponding  transformation linearizing an isochronous center.
Such systems possess many good properties, such as commutativity, transversality, and so on.
Besides, 
their topological structures near the points at infinity  on $L_h$ can also provide for us a lot of information for the isochronicity of system (\ref{H-C}).
Letting $H_{n+1}(x,y)$ be the highest degree part of  $H(x,y)$,
 we have the following necessary condition for isochronicity:
\begin{theorem}\label{non-iso}
For system (\ref{H-C}),	 if the origin  is an isochronous center,
	then $H_{n+1}$ must have a factor with multiplicity  no less than $ (n+1)/2$.
\end{theorem}

In  this paper, we will also show an interesting relation between Gavrilov's question and the following conjecture, which was
claimed by X. Jarque and J. Villadelprat in \cite{J-V}, on real systems (\ref{H-C}), i.e., $(x,y)\in \mathbb{R}^2,$  and $t\in \mathbb{R}$.

\begin{conjecture}[Jarque-Villadelprat conjecture]
	If $n$ is even, then the real system (\ref{H-C}) has no isochronous centers.
\end{conjecture}
At present, this conjecture is still open
and a recent development  can be found in \cite{Cre-Pala}.
The following theorem indicates that if the Jarque-Villadelprat conjecture is not true,
then the Gavrilov's question must have a negative answer for  such real systems.
\begin{theorem}\label{G-J-V}
For any isochronous center of a real system (\ref{H-C}) with even $n$, the corresponding vanishing cycle  can not be homologous to  zero on the projective closure of the complexification of a generic real level curve.
\end{theorem}

The paper is organized as follows.
We shall first introduce some properties on the commuting real differential systems(or real vector fields)  induced by system (\ref{H-C}) and provide a powerful technique to extend the  transformation linearizing an isochronous center.
Then we give the detailed proof of the main results and some applications.

\section{Commuting real systems}\label{continu}

Note that if  the origin is an isochronous  center of Morse type for system  (\ref{H-C}),
then there exists an analytic area-preserving transformation(see, e.g, \cite{AGR,L-R,M-V})
$$\Phi:\ (x,y)\mapsto (u(x,y),v(x,y))$$
changing system  (\ref{H-C})
to a linear system
\begin{eqnarray}\label{H-C-L}
\left(
\begin{array}{c}
	\frac{d u}{d t} \\
	\frac{d v}{d t} \\
\end{array}
\right)
=
\left(
\begin{array}{r}
	-v \\
	u \\
\end{array}
\right),
\end{eqnarray}
here we say $\Phi $ is area-preserving is equivalent to say
its Jacobian determinant $\det (J(\Phi)) \equiv 1$, where 
$$J(\Phi)=\left( \begin{array}{cc}
	\frac{\partial u}{\partial x}	& \frac{\partial u}{\partial y} \\
	\frac{\partial v}{\partial x}	& \frac{\partial v}{\partial y}
\end{array}\right).$$
Generally speaking, $\Phi$ is only well defined  in a small neighborhood of the origin   $\mathcal{N} (O)\subseteq \mathbb{C}^2$.

By taking advantage of constant Jacobian determinant,
one can construct another   complex system in $\mathcal{N} (O)$ as follows
\begin{eqnarray}\label{g-C}
\left(
  \begin{array}{c}
    \frac{d x}{d t} \\
    \frac{d y}{d t} \\
  \end{array}
\right)
=
(J^{T}J)^{-1}\left(
  \begin{array}{c}
   \frac{\partial H}{\partial x} \\
  \frac{\partial H}{\partial y} \\
  \end{array}
\right), 
\end{eqnarray}
which can be also  linearized to a linear system
\begin{eqnarray}\label{g-C-l}
\left(
\begin{array}{c}
	\frac{d u}{d t} \\
	\frac{d v}{d t} \\
\end{array}
\right)
=
\left(
\begin{array}{r}
	u \\
	v \\
\end{array}
\right)
\end{eqnarray}
by the same transformation $\Phi$, for the reasons that 
 \begin{eqnarray*}
	J^{-1}= \left(
	\begin{array}{rr}
		\frac{\partial v}{\partial y} & -\frac{\partial u}{\partial y}\\
		-\frac{\partial v}{\partial x} & \frac{\partial u}{\partial x}
	\end{array}
	\right)
\end{eqnarray*}
and 
 \begin{eqnarray*}
	(J^{-1})^{T}\left(
	\begin{array}{c}
		\frac{\partial H}{\partial x} \\
		\frac{\partial H}{\partial y} \\
	\end{array}
	\right)=\left(
	\begin{array}{rr}
		\frac{\partial v}{\partial y} & -\frac{\partial v}{\partial x}\\ 
		-\frac{\partial u}{\partial y} & \frac{\partial u}{\partial x}
	\end{array}
	\right)
	\left(
	\begin{array}{c}
		\frac{\partial H}{\partial x} \\
		\frac{\partial H}{\partial y} \\
	\end{array}
	\right)=
	\left(
	\begin{array}{r}
		u \\
		v \\
	\end{array}
	\right).
\end{eqnarray*}
Consequently,  systems (\ref{H-C}) and (\ref{g-C-l}) 
induce the following four real differential systems(see, e.g. \cite{Cama-LN-Sad-1}) 
by taking $(x,y)\in \mathbb{C}^2\cong \mathbb{R}^4$ but $t\in \mathbb{R}$:
\begin{eqnarray*}
V:
\left(
\begin{array}{c}
	\frac{d x}{d t}\\
	\frac{d y}{d t}
\end{array}
\right)
=
\left(
\begin{array}{r}
-\frac{\partial H}{\partial y}\\
\frac{\partial H}{\partial x}
\end{array}
\right),\
{\rm i}V:
\left(
\begin{array}{c}
	\frac{d x}{d t}\\
	\frac{d y}{d t}
\end{array}
\right)
=
\left(
\begin{array}{r}
- {\rm i}\frac{\partial H}{\partial y}\\
 {\rm i}\frac{\partial H}{\partial x}
\end{array}
\right), 
\end{eqnarray*}
and
\begin{eqnarray*}
	V_g:
	\left(
	\begin{array}{c}
		\frac{d x}{d t} \\
		\frac{d y}{d t} \\
	\end{array}
	\right)
	=
	(J^{T}J)^{-1}\left(
	\begin{array}{c}
		\frac{\partial H}{\partial x} \\
		\frac{\partial H}{\partial y} \\
	\end{array}
	\right), \
{\rm i}V_g:
\left(
  \begin{array}{c}
    \frac{d x}{d t} \\
    \frac{d y}{d t} \\
  \end{array}
\right)
=
(J^{T}J)^{-1}\left(
  \begin{array}{c}
  {\rm i}  \frac{\partial H}{\partial x} \\
  {\rm i} \frac{\partial H}{\partial y} \\
  \end{array}
\right),
\end{eqnarray*}
where ${\rm i}^2=-1$.
They can be transformed to the following four real linear systems simultaneously by  the same $\Phi$ respectively:
\begin{eqnarray*}
V \hookrightarrow V_{\ast}: \left(
\begin{array}[2]{c}
\frac{d u}{d t}\\
\frac{d v}{d t}
\end{array}
\right)
=\left(
\begin{array}[2]{r}
-v\\
u
\end{array}
\right), \
{\rm i}V \hookrightarrow {\rm i}V_{\ast}: 	\left(
\begin{array}[2]{c}
\frac{d u}{d t}\\
\frac{d v}{d t}
\end{array}
\right)
=\left(
\begin{array}[2]{r}
-{\rm i}v\\
{\rm i}u
\end{array}
\right), 
\end{eqnarray*}
and
\begin{eqnarray*}
V_g \hookrightarrow 	V_{g\ast}: 	\left(
	\begin{array}[2]{c}
		\frac{d u}{d t}\\
		\frac{d v}{d t}
	\end{array}
	\right)
	=\left(
	\begin{array}[2]{r}
		u\\
		v
	\end{array}
	\right), \
{\rm i}V_{g} \hookrightarrow  {\rm i}V_{g\ast}: 	\left(
\begin{array}[2]{c}
\frac{d u}{d t}\\
\frac{d v}{d t}
\end{array}
\right)
=\left(
\begin{array}[2]{r}
{\rm i}u\\
{\rm i}v
\end{array}
\right).
\end{eqnarray*}

Letting $u=u_1+ \mathrm{i }u_2$ and $v=v_1 + \mathrm{i} v_2$ and
regarding $\mathbb{C}^2 \cong \mathbb{R}^4=\{(u_1, u_2, v_1, v_2) \}$,
 the coefficient matrices of $V_{\ast}$, ${\rm i}V_{\ast}$, $V_{g\ast}$ and  ${\rm i}V_{g\ast}$
are respectively
\begin{eqnarray*}
\begin{array}{lcclcc}
M_1
&=&\left(
   \begin{array}{rr}
     0 & -I_2\\
     I_2 & 0\\
   \end{array}
 \right)
,\
&M_2
&=&\left(
   \begin{array}{rr}
     0 & -E_2\\
     E_2 & 0\\
   \end{array}
 \right),\\ \\
M_3
&=&\left(
   \begin{array}{rr}
     I_2 & 0\\
     0 & I_2\\
   \end{array}
 \right),\
 & M_4
 &=&\left(
 \begin{array}{rr}
 	E_2 & 0\\
 	0 & E_2\\
 \end{array}
 \right),
  \end{array}
\end{eqnarray*}
 where
\begin{eqnarray*}
I_2
=\left(
   \begin{array}{cc}
     1 & 0\\
     0 & 1\\
   \end{array}
 \right)
,\
E_2
=\left(
   \begin{array}{cc}
     0 & -1\\
     1 & 0\\
   \end{array}
 \right).
\end{eqnarray*}
Obviously we have  $$M_{i}M_{j}=M_{j}M_{i}, \ \forall i,j=1,2,3,4.$$
Due to that $\Phi$ is a diffeomorphism, one can get  the following important properties for vector fields $V$, ${\rm i}V$, $V_g$, and $\ {\rm i}V_g$:

\begin{enumerate}
	
	\item they are commutative  pairwise everywhere in $\mathcal{N} (O)$, i.e., as real vector fields, 
	the  Lie bracket of any two of them vanishes. So for any two points $p_1,p_2$ in $\mathcal{N} (O)$ except $O$, it takes the same time along any two continuous paths connecting $p_1$ and $p_2$ consisting of finitely many trajectories of those vector fields.
	
	\item their trajectories are transversal pairwise everywhere on $\mathcal{N} (O)-L_0$; while on $L_0$, $V$(resp. ${\rm i}V$) coincides with ${\rm i}V_{g}$(resp. $-V_{g}$)
	on one of two branches near $O$ and with $-{\rm i}V_{g}$(resp. $V_{g}$) on the other one;
	
	\item  the domain in which $V_{g}$ and ${\rm i}V_{g}$ can be well defined is the same to the domain of $\Phi$, but $V$ and ${\rm i}V$ are well defined on the whole complex plane $\mathbb{C}^2$;
	
	\item  near the origin, all of the orbits of system $V$ are closed; 
	on the contrary,  system ${\rm i}V $ does not have any closed orbits  in $\mathcal{N} (O)$;
	
	\item the trajectories of systems $V$ and ${\rm i}V$  are both tangent to $L_h$ everywhere,
	so their  restrictions, denoted by  $V_h$ and ${\rm i}V_h$,  are two real systems defined well on $L_h$.
\end{enumerate}

Denote by $\varphi(\cdot,t)$(resp. $ {\rm i}\varphi,\ \varphi_g, \ {\rm i}\varphi_g$,
$\varphi_{\ast},\ {\rm i}\varphi_{\ast},\ \varphi_{g\ast},$ and $ {\rm i}\varphi_{g\ast}$)
the flow map induced by  $V$(resp. ${\rm i}V$, $V_g,\ {\rm i}V_g$, $V_{\ast}$, ${\rm i}V_{\ast}$, $V_{g\ast},$ and $ {\rm i}V_{g\ast}$), i.e., for any given point $p$, $\varphi(p,t)$ takes the value at time $t$ of the solution of equations $V$ with initial value $p$ at  $t=0$. 
The commutativity between those systems means that each one of the flow maps above  preserves the orbits of any other  system in the domain of $\Phi$.
We shall take advantage of  this observation to extend the domain of $\Phi$ to a bigger one than $\mathcal{N} (O)$. 
Without loss of generality,  we assume $\mathcal{N} (O)$ is a sufficiently small and homeomorphic to a open
ball $\{(x,y)\in \mathbb{C}^2: \ \left| x \right| ^2 +\left| y \right| ^2<\epsilon\}$ centered at $O$ with
radius $\epsilon$. Denote by $H$ the following map
\begin{eqnarray*}
	\begin{array}{lrll}
			H: & \mathbb{C}^2&\longrightarrow &\mathbb{C}\\
			& (x,y) &\longmapsto & H(x,y).
	\end{array}	
\end{eqnarray*}


\noindent \textit{Continuation technique for $\Phi$:}

For a  closed orbit $ \sigma$ of $V$ in $\mathcal{N}(O)$ such that $H(\sigma)\neq 0$,   
and two sufficiently small number $t_1,t_2$, the space 
$$\Gamma_{\sigma}\triangleq \cup_{0\leq s\leq t_1}{\rm i}\varphi_{g}\left(\cup_{0\leq t\leq t_2} \varphi_{g}(\sigma,t),s\right)$$
is a real  3-dimension sub-manifold of $\mathcal{N}(O)$ and 
transversal  to ${\rm i} V$ at every point. 
Then $\Phi(\Gamma_{\sigma})$ is also   a real  3-dimension sub-manifold of $\Phi(\mathcal{N}(O))$ and 
transversal  to ${\rm i} V_{\ast}$ at every point.

Along the trajectories of ${\rm i} V$ passing through $\Gamma_{\sigma}$ in $\mathcal{N}(O)$, the transformation $\Phi $  can be expressed by the flow map ${\rm i}\varphi$ as follows: for any point $p\in \Gamma_{\sigma},$ and any sufficiently small $t$, we have 
\begin{equation}\label{con-teq}
	\Phi({\rm i}\varphi(p, t)) = {\rm i}\varphi_{\ast}(\Phi(p), t).
\end{equation}

Clearly the vector fields ${\rm i} V $ and $ {\rm i} V_{\ast}$ are well defined globally, the above equation can be extended to a larger interval $I$ for time $t\in \mathbb{R}$ such that $ {\rm i}\varphi(p, t) \not \in \mathcal{N}(O)$, if ${\rm i} V $ satisfies the following two conditions:
\begin{enumerate}
  \item[C1.] the trajectories of  ${\rm i} V $ could not return into the domain where $\Phi$ has already been defined well;
  \item[C2.] there is  no point $P$ at infinity such that $ {\rm i} \varphi(p, t)$ tends to $P$ as $t$ tends a finite  moment $t_0$ for some a  point $p_0\in \Gamma_{\sigma}$.
\end{enumerate}
If the trajectories of  ${\rm i} V $ from a open subset of $\Gamma_{\sigma}$ go to a point at infinity when $t$ tends $\infty$, then the interval $I$ for those points can be $[0,+\infty)$ or $(-\infty,0]$.
While if C1 holds but C2 not, then $I$ can only be  $[0,t_0)$ or $(t_0,0]$ at such a point $p_0$(see Figure \ref{figures-Con}).

\begin{figure}[ht] 	
 	\centering
 	\includegraphics[]{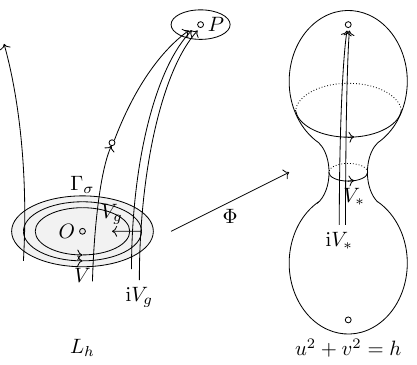}
 	\caption{The continuation of $\Phi$}
 	\label{figures-Con}
 \end{figure}	
    
Noticing that the vector fields $ V $ and $  V_{\ast}$ are also well defined globally, we can also perform the above operation along the trajectories of $ V $ if it satisfies the conditions C1 and C2. 

In a word, one can extend the transformation $\Phi $ to an open domain  $\mathcal{D}$ as big as possible according to the above operation along the trajectories of $V$ and ${\rm i} V $. Although $\mathcal{D}$ may be much bigger than $\mathcal{N}(O)$,  we have  $\mathcal{D}\cap L_h$ is still homeomorphic to $\mathcal{N}(O)\cap L_h$ for any $h$ sufficiently close to $0$.


\section{Points at infinity}
To prove the main results, we still need to know some information about the  points at infinity on $L_h$.
It is better to deal with it  in the projective space $\mathbb{CP}^2$.
Assume  the projective closure $\overline{L_h}$  are defined by the following homogeneous equations
$$\sum_{k=2}^{n+1}z^{n+1-k}H_{k}(x,y) -h z^{n+1}=0,\  [x:y:z]\in \mathbb{CP}^2,$$
where $H_k$ represents the homogeneous part of  degree $k$ of  $H(x,y)$.
For a generic value $h$, the set of singularities on $\overline{L_h}$,
denoted by $\Sigma_h$, consists of only some points at infinity on $L_h$.
The Riemann surface of $L_h$ coincides with the resolution of  $\overline{L_h}$ by a birational map.
Generally speaking, 
The algebraic curve $\overline{L_h}$ may have  more than one connected branches near a  point  $P\in \Sigma_h$.
The number of such branches   is equal to the number of essentially different Puiseux expressions  associated to $P$(see, e.g., \cite{Kir}).

Rewriting  the homogeneous part $H_{n+1} (x,y)$ of degree $n+1$  as follows:
\begin{align}\label{H_n}
	H_{n+1} (x,y)= \prod_{i=1}^{N} \left(\alpha_i x - \beta_i y\right)^{n_i},
	\ \ n_i \geq 1, \ \ \sum_{i=1}^{N}n_i =n+1,
\end{align}
where $\alpha_i,\ \beta_i \in \mathbb{C}$ such that $\alpha_i :\beta_i\neq \alpha_j :\beta_j$ if $i\neq j,$
the projective coordinate of a point $P^{i}\in \Sigma_{h}$ can be  
represented by $[\beta_i :\alpha_i : 0]$.
Up to a projective change of coordinates, 
we can always assume its projective coordinate is $[1:0:0]$.
Then it is convenient to adopt a pair of new affine coordinates $(X,Y)$, where $$X=\frac{1}{x},\ Y=\frac{y}{x},$$
and the Puiseux expressions near $P^{i}$ are totally determined by
 the Puiseux expressions of  equation
\begin{eqnarray}\label{Pui-eqn}
	H_{h}^{\ast}(X,Y)=X^{n+1} H\left(\frac{1}{X},\frac{Y}{X}\right)-hX^{n+1}=0
\end{eqnarray}
near the origin.
According to the classical theory of Puiseux(see, e.g., (\cite{Fischer})), 
each branch of an algebraic curve near a singularity can be  parameterized  by a Puiseux series of the following form.
\begin{lemma}[Puiseux]\label{Pui-par}
	If $H_{h}^{\ast}(0,0)=0$ and $H_{h}^{\ast}(0,Y)\neq 0$,
	then there exist  numbers $\mathsf{p}, \mathsf{q}\in \mathbb{Z}_{+}$,
	a parameter $s\in \mathbb{C}$, and a a holomorphic function $\rho(s)=s^{ \mathsf{q}} (c_0+\sum_{i= 1}^{+\infty}c_i s^i),\ c_0 \neq 0,$
	such that $H_{h}^{\ast}(s^{\mathsf{p}}, s^{ \mathsf{q}} \rho(s))= 0$ for all $s$ in a neighbourhood of $0$.
\end{lemma}

In general the coefficients $\{c_i\}$ may depend on $h$ on different level curve $L_h$, 
so sometimes we replace $\rho(s)$ with $\rho(s,h)$ to emphasize it.
Taking the Puiseux parameterization $x=s^{-\mathsf{p}}, \ y=s^{\mathsf{q}-\mathsf{p}}\rho(s,h)$ into system (\ref{H-C}),
 we obtain a complex 1-dimension ordinary differential equation 
\begin{eqnarray}\label{vf-inf}
\frac{ds}{dt}= \frac{ s^{\mathsf{p}+1}}{\mathsf{p}} \frac{\partial H}{\partial y}(s^{-\mathsf{p}},s^{\mathsf{q}-\mathsf{p}}\rho(s,h))
= \lambda s^{\mathsf{k}}+o(s^{\mathsf{k}})
\end{eqnarray}
on a branch of   $\overline{L_h}$ near $P^{i}$, 
where $\lambda\not =0,\  \mathsf{k}\in \mathbb{Z}.$
Then the real  systems $V_h$ and ${\rm i}V_h$ are changed to the following forms respectively under this parameterization:
\begin{eqnarray}\label{vf-inf-V}
V_h:\ \ 	\frac{ds}{dt}=\lambda s^{\mathsf{k}}+o(s^{\mathsf{k}}),\  t\in \mathbb{R},
\end{eqnarray}
and
\begin{eqnarray}\label{vf-inf-iV}
{\mathrm{i}}V_h: \ \ 	\frac{ds}{dt}={\mathrm{i}} \lambda s^{\mathsf{k}}+o(s^{\mathsf{k}}),\  t\in \mathbb{R}.
\end{eqnarray}
Their topological structures 
can be classified into the following four classes according to the value of $\mathsf{k}$ near $s=0$:
\begin{itemize}
	\item $\mathsf{k}>1$.  The orbits  of real system (\ref{vf-inf-V})(or system (\ref{vf-inf-iV})) form  $2(\mathsf{k}-1)$ petals in a sufficiently small neighborhood of $s=0$, any one of them is tangent  to a separatrix of the petals at $s=0$.
	
	\item $\mathsf{k} = 1$. If $\lambda$ is a pure imaginary number, the point $s=0$ is of center-focus type;
while for $\lambda \in \mathbb{R}$, it is a node,
and for other numbers, it is a focus.
	
	\item $\mathsf{k} =0 $. The point $s=0$ is  not a singularity for real systems (\ref{vf-inf-V}) and (\ref{vf-inf-iV}).
	
	\item $\mathsf{k} < 0$.  The  system (\ref{vf-inf-V})(or  system (\ref{vf-inf-iV})) has a saddle structure in a sufficiently small neighborhood of $0$ except $s=0$.
                                           
\end{itemize}

\begin{remark}\label{fin-inf}
It should be pointed out that  there may exist an orbit of  system (\ref{vf-inf-V})  such that  it can reach the origin $s=0$ at a finite  moment from a fixed point $s\not = 0$.  
It is not difficult to see this phenomenon occurs only in  the cases  $\mathsf{k} =0$ and $\mathsf{k} < 0$, and in the latter one such an orbit is just the separatrix of the saddle.
\end{remark}

The Puiseux parameterizations can be determined completely by the so-called {\emph{ Newton polygon}} of the singularity.
Given an irreducible polynomial $F(X,Y)=\sum_{k,l}b_{kl}X^k Y^l$ with
$F(0, 0) = 0$, denote by $\Lambda(F)$ the carrier of $F(X,Y)$, i.e.
$\Lambda(F) = \{(k, l) \in \mathbb{Z}^2 \ \vert \  b_{kl}\neq 0\}.$
Assuming that $Q_1,Q_2\in \mathbb{R}^2$, let 
$$[Q_1,Q_2]=\{\sigma Q_1 +(1-s)Q_2\ \vert \ 0\leq s\leq 1\}$$
be the straight line segment from $Q_1$ to $Q_2$.
Consider the convex subset $A$ on $\mathbb{R}^2$ consisting of those $(X,Y)\in \mathbb{R}^2$ such that $X\geq X_0$ and $Y\geq Y_0$  for some $(X_0,Y_0)\in [Q_1,Q_2]$ where $Q_1,Q_2\in \Lambda(F)$.
\begin{definition}[Newton Polygon]

The boundary of set $A$ excluding the axes is called the {\emph{Newton polygon}}  of $F(X,Y)$ at the origin, which consists  of only finitely many straight line segments.
\end{definition}

\section{Important lemmas}

In this section,  we first prove the following important lemmas.
It is not difficult to see the vanishing cycle $\gamma_h$ can be represented by a given closed orbit of system $V_h$ near the origin(we still denote this orbit by $\gamma_h$).

\begin{lemma}\label{non-c-o}
If the origin is an isochronous  center of system (\ref{H-C}), 
then every orbit of ${\mathrm i} V_h$ passing through a point on  $\gamma_h$ is not closed for any $h$ sufficiently close to $0$.
\end{lemma}

\begin{proof}
Suppose otherwise, i.e., suppose there exists a point $p_0 \in \gamma_h$ such that the orbit of ${\mathrm i} V_h$ passing through $p_0$ is closed.
Then by the commutativity between $V$ and ${\mathrm i} V$, 
there exists a sufficiently small neighborhood $\mathcal{N}_{p_0}\subset \Gamma_{\gamma_h} \subset\mathcal{N}(O)$ of $p_0$ such that for any $p\in \mathcal{N}_{p_0}$, 
the orbit of ${\mathrm i} V$ passing through $p$ is also closed.

Consider the inverse transformation $\tilde{\Phi}=\Phi^{-1}$ that  also has a constant Jacobian  
determinant $1$
in the domain $\Phi(\mathcal{N}(O))$ on the $(u,v)$-plane.
Since ${\mathrm i} V_{\ast}$ satisfies the conditions C1 and C2, by using the same continuation technique introduced in Section \ref{continu},
 we can extend $\tilde{\Phi}$ from $\Phi(\mathcal{N}(O))$ to a bigger domain $\tilde{\mathcal{D}}$ along the trajectories of ${\mathrm i} V_{\ast}$   
 by the following equation
$$\tilde{\Phi}({\rm i}\varphi_{\ast}(q, t)) \triangleq {\rm i}\varphi(\Phi^{-1}(q), t),\  \forall q\in \Phi({\mathcal{N}_{p_0}}),$$
such that $\tilde{\Phi}(\tilde{\mathcal{D}})$ covers all closed orbits of ${\rm i} V$ passing through $\mathcal{N}_{p_0}$(see Figure \ref{figures-no-c}).

	\begin{figure}[ht] 	
 	\centering
 	\includegraphics[]{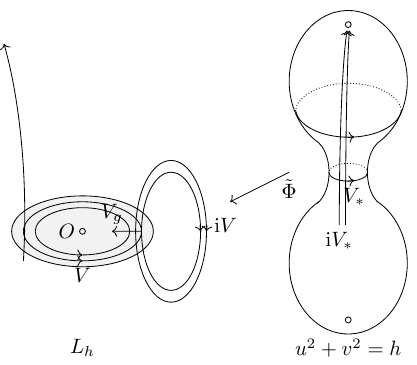}
 	\caption{The continuation of $\tilde{\Phi}$}
 	\label{figures-no-c}
 \end{figure}	

In the domain $\tilde{\Phi}(\tilde{\mathcal{D}})$,
the vector field $V_{g}$ is well defined and commuting with ${\mathrm i} V$. 
This implies that the periods of those closed orbits of  ${\mathrm i} V$ are the same for any 
$h\in H(\mathcal{N}_{p_0})$. 
The above operation is valid for any $h$ sufficiently close to $0$.  
So we get a series of closed orbits of ${\mathrm i} V$ with the same period as $h\rightarrow 0$ along a trajectory of $V_{g}$, whose lengths tend to $0$ since $\vert {\mathrm i} V \vert \rightarrow 0$ when $h\rightarrow 0$. 
This means that such a  closed orbit    also represents  the vanishing cycle of the isochronous center,
which leads a contradiction, because the origin is of Morse type having only one vanishing cycle and the intersection number of two closed orbits of $V  $ and $\mathrm{i}V $ respectively is equal to $1$ so that they can not represent the same one cycle in $\mathcal{H}_1(L_h,\mathbb{Z})$.
Thus the lemma holds.
\end{proof}

By this lemma, we have the following immdiately.

\begin{lemma}\label{inf-center}
	If the origin is an isochronous center of system (\ref{H-C}), then 
there exists  a subset $\gamma^{1}_{h}\subset \gamma_h$ consisting of at most finitely many points  such that:
	
	\begin{enumerate}
		\item for any $p\in \gamma^{1}_{h}$,
	${\rm i }\varphi(p, t) $ tends to a point $P^{i}$ at infinity as $t$ tends to some a finite time $t_0$,  and $V_h$ has a form (\ref{vf-inf-V}) with  number $\mathsf{k}\leq 0$ on one of  the branches of $\overline{L_h}$ near $P^{i}$; 
	
	\item for any $p\in \gamma_h - \gamma^{1}_{h}$,
	${\rm i }\varphi(p, t) $ tends to a point $P^{i}$ at infinity as $t\rightarrow \pm \infty$,  and 
	$V_h$ has a form (\ref{vf-inf-V}) with  number $\mathsf{k}\geq 1$ on one of  the branches  of $\overline{L_h}$  near $P^{i}$.
	\end{enumerate}

\end{lemma}

\begin{proof}
By Lemma  \ref{non-c-o}, if the orbit $\delta_p$ of ${\mathrm i} V_h$ passing through a point   $p\in \gamma_h$ can not tend any point at infinity, then there remains two possible cases:

\begin{itemize}
  \item $\delta_p$ tends a closed orbit $\delta_0$ of ${\mathrm i} V_h$. If such a $\delta_0$ exists, then it is isolated or semi-isolated. However,
by the commutativity between $V$ and ${\mathrm i} V$, there exist  annuli such that $\delta_0$ is not the boundary.

  \item $\delta_p$ is ergodic on  a subset of $L_h$. If so, we can also extend the transformation $\tilde{\Phi}$ to a domain $\tilde{\mathcal{D}}$ such that $\tilde{\Phi}(\tilde{\mathcal{D}})$ covers $\delta_h$ as shown in the above lemma along $\delta_p$(in fact, we only need to do this on $L_h$). One can choose  a trajectory of $l_h$ for $V$ such that  $\delta_p \cap l_h$ is dense in $l_h$.
Noticing that ${\rm i} V_{\ast}$ on the curve $C_h$ defined by $u^2 +v^2=h$ is integrability, there is a non trivial analytic first integral $\tilde{G}$ defined on $C_h$ such that $\tilde{G}(\tilde{\Phi}^{-1}(\delta_h))$ is a constant.
Defining a function $G(p)\triangleq \tilde{G}(\tilde{\Phi}^{-1}(p))$ for  $p\in  \tilde{\Phi}(\tilde{\mathcal{D}}) $, it is  a non trivial  analytic first integral for  ${\mathrm i} V$ such that $G(l_h)$ is not a constant. However, $G$ is a constant on a dense subset $\delta_p \cap l_h$  of $l_h$, which implies $G(l_h)$ should be also a constant. This is a contradiction.

\end{itemize}

Finally, every orbit of ${\mathrm i} V_h$ passing through a point   $p\in \gamma_h$
can only tend to a point $P^i$ at infinity on one of  the branches of $\overline{L_h}$ near $P^{i}$.
According to the arguments in Remark \ref{fin-inf},
if the number $\mathsf{k}\leq 0$ for $P^i$, then  ${\rm i }\varphi(p, t)$ will reache at  $P^{i}$ at some a finite  moment  $t_0$.
In addition,
due to that the numbers of points at infinity and separatrices of the saddles  are both finite,
the number of such points $p$ are also finite. 
The lemma is proved.	
\end{proof}

Below we shall show that, under the assumption of Theorem \ref{h-0-c}, in the second case of the above lemma, the number $\mathsf{k}$ must be equal to $1$.

\begin{lemma}\label{k1}
Under the assumption of Theorem \ref{h-0-c},
if   a point $P^{i}$  at infinity on $L_h$ is the limit of ${\mathrm i} \varphi(p,t)$ as $t\rightarrow +\infty$(or $-\infty$) for some a point $ p\in \gamma_h$  on one of the   branches near $P^{i}$, 
then  
the number $\mathsf{k}=1$ for corresponding system (\ref{vf-inf}), 
and the orbits of $V_h$ are closed encircling $P^i$.
\end{lemma}

\begin{proof}

Suppose system (\ref{H-C}) has a Hamiltonian $H=(x^2 +y^2)/2+h.o.t$ and $P^{i}$  has a projective coordinate $[\beta_i:\alpha_i:0]$,
by a linear change of coordinates $(x,y)\mapsto (x_1,y_1) =(-\left(\overline{\beta_i} x+\overline{\alpha_i} y \right)/r_i,\left(\alpha_i x - \beta_i y\right)/r_i)$, where $r_i=\sqrt{| \alpha_i | ^2+| \beta_i | ^2}$, its coordinate can be changed to $[1:0:0]$.
If the linearization transformation $\Phi$ maps $(x,y)$ to $(u,v)$, by taking a linear change of coordinates
 $$\left(\begin{matrix}
  u_1 \\
 v_1
\end{matrix}\right)=
\frac{1}{\sqrt{2}r_i}\left(\begin{matrix}
  \overline{\alpha_i}-{\rm i}\beta_i & -{\rm i}\overline{\alpha_i}+\beta_i \\
 -{\rm i}\alpha_i -\overline{\beta_i} & \alpha_i +{\rm i} \overline{\beta_i}
\end{matrix}\right)
\left(\begin{matrix}
  u \\
 v
\end{matrix}\right),$$ then one of the points at infinity also has  a coordinate $[1:0:0]$ on $(u_1,v_1)$-plane, i.e., the Hamiltonian function has the form $u_1 v_1 $.

Under the conditions of the lemma, the equation (\ref{con-teq}) holds for $I=[0,+\infty)$ and a sufficiently small 
neighborhood of $p$ in $\Gamma_{\gamma_h}$, i.e.,
the domain $\mathcal{D}$ where $\Phi$ is well defined can be 
sufficiently close to $P^{i}$ along the orbits of ${\mathrm i} V$.

 We take the coordinates of Puiseux parameters $(s,h)$ near  $P^{i}$  in $\mathcal{D}$, and  the coordinates $(\tilde{s},h)$ near $[1:0:0]$ in $\Phi(\mathcal{D})$.
Then $\Phi $ induces a map $ \Psi $ from an open set $\mathcal{P}$ in the $(s,h)$ plane    to $(\tilde{s},h)$ plane $\tilde{\mathcal{P}}$,
so that the following diagram is commutative.

\begin{equation}
		\begin{tikzcd}
		\mathcal{D} \ar[d,"\Phi"']   & \ar[l, "\mathrm{R}"']\mathcal{P}  \ar[d,"\Psi"] \\
\Phi(\mathcal{D})   & \ar[l, " \tilde{\mathrm{R}}"] \tilde{\mathcal{P}}
		\end{tikzcd}\ \ \ \ \ \ \ \ \ 		
			\begin{tikzcd}
		(x_1,y_1) \ar[mapsto, d,"\Phi"']   & \ar[mapsto,l, "\mathrm{R}"'](s,h)  \ar[mapsto,d,"\Psi"] \\
(u_1,v_1)   & \ar[mapsto, l, " \tilde{\mathrm{R}}"] (\tilde{s},h)
		\end{tikzcd}
\end{equation}
where 
\begin{align}
	\begin{array}{llcl}
		\mathrm{R}: &(s,h)&\mapsto & (s^{-\mathsf{p}},\ s^{\mathsf{q}-\mathsf{p}}\rho(s,h))\\
		\tilde{\mathrm{R}}: &(\tilde{s},h)&\mapsto & (\tilde{s}^{-1},h\tilde{s})
	\end{array}.
\end{align}
are Puiseux parameterizations respectively, and  
\begin{align}
	\begin{array}{llcl}
         \Phi: & (x_1,y_1)& \mapsto & (u_1,v_1)\\
         \Psi: &(s,h)&\mapsto & (\tilde{s},h)
	\end{array}.
\end{align}

Denoting by $\Psi(s,h)=(\psi(s,h),h)$, we have 
\begin{equation}\label{ss}
	\frac{\partial \psi}{\partial s} \frac{d s}{d t}= \frac{ d \tilde{s}}{dt }=\tilde{s}=\psi(s,h),
\end{equation}
that is,
$$ \frac{\partial \ln \psi}{\partial s} = \frac{  1}{d s/dt }.$$

Recall the system (\ref{vf-inf}) is the following
 $$\frac{d s}{d t}=\lambda s^{\mathsf{k}}+o(s)\triangleq \tau(s),$$
and the period $1$-form $dt =d s/ \tau(s)$ can not have a pole at $P^i$ with zero residue, i.e. in the Laurent series of $1/\tau(s)$, the coefficient of $1/s$ is  a nonzero number $c_0$. Thus, $\psi(s,h)$ can be expressed in $s$ as follows:
$$\psi(s,h)= se^{c_0+\tau_1(s)+\tau_{2}(s)},$$
where $\tau_{1}(s)=\sum_{j\geq 1}c_{1j}s^{j}$, $\tau_{2}(s)=\sum_{-\mathsf{k}+1\leq j\leq -1}c_{2j}s^{j}$, $c_{2(-\mathsf{k}+1)}=1/\lambda$,
and equation (\ref{ss}) becomes 
$$e^{c_0+\tau_1(s)+\tau_{2}(s)} (1+\tau_1'+\tau_2')(\lambda s^{\mathsf{k}}+o(s))=se^{c_0+\tau_1(s)+\tau_{2}(s)}, $$
which  implies that $\tau_2 =0$, $\mathsf{k}=1$. 
Furthermore,  $\psi(s,h)$ can be analytically extended  to  a sufficiently small disc encircling $(0,h)$.
Clearly the orbits of $V_{\ast}$ are all closed, so are the orbits of vector fields on $(\tilde{s},h)$ and $(s,h)$ planes induced by $V$ and $V_{\ast}$ respectively.
Besides, due to that the Puiseux parameterization $\mathrm{R}$ is a finitely many cover mapping near $P^i$, 
the orbits of $V_h$ are also closed.
\end{proof}

\begin{remark}
	In the proof of the above lemma, the conclusion, that the function $\psi(s,h)$ can be analytically extended  to the origin of $(s,h)$ plane, does not mean the transformation $\Phi$ can be also   analytically extended  to $P^i$, one of the reasons is the inverse function of Puiseux parameterization $\mathrm{R}$ is usually multi-valued near $P^i$. 
\end{remark}

This lemma tells us the period $1$-form $\omega$ of 
system (\ref{H-C}) has at least one pole at a point $P^i$ at infinity on a branch of $\overline{L_h}$ near $P^i$. The following lemma  will show that the multiplicity of such a point $P^i$ can not be too low, i.e., we have 

\begin{lemma}\label{n/2}
	If  $\omega$ has a pole at a point $P^i$ at infinity on a branch of  $\overline{L_h}$  near $P^i$, then the multiplicity $n_i$ of $P^i$ satisfies $n_i \geq (n+1)/2$.
\end{lemma}

\begin{proof}
We still assume the projective coordinate of $P^{i}$ is $[1:0:0]$.	
 Let $\{(k_i,l_i), \ i=0,\cdots, r\}$ be the vertex set of the Newton polygon of $H_{h}^{\ast}(X,Y)$ near the origin,
	where $l_0\geq l_1\geq \cdots\geq l_r=0$, $0=k_0 \leq k_1\leq\cdots\leq k_r$.
Denoting  by $N_i=\min\{\mathsf{p}k_i+\mathsf{q} l_i,\ i=0,\cdots,r\},$ the minimum of $\mathsf{p}k_i+\mathsf{q} l_i$,
	there exists a straight line on $(k,l)$-plane
	$$ \mathcal{L}:\ \  \mathsf{p}k+\mathsf{q} l=N_i$$ passing all the points contained in $\{(k_i, l_i): \ \mathsf{p}k_i+\mathsf{q} l_i =N_i\}$.
We define the \emph{Newton principal polynomial} $g_{N_i}(X,Y)$ by  the following
	$$g_{N_i} (X,Y)=\sum_{(k,l)\in\mathcal{ L}}b_{k,l} X^k Y^l,$$
	where $b_{k,l} $ is the  coefficient of term $X^k Y^l$ of $H_{h}^{\ast}(X,Y)$.
	
 Taking the Puiseux parameterization $x=s^{-\mathsf{p}}, \ y=s^{\mathsf{q}-\mathsf{p}}\rho(s)$  into $\omega$,
 we get that
 $\omega$ has a pole at $P^i$ if and only if 
 \begin{eqnarray}\label{pole}
 	Ord \left(\frac{s^{-\mathsf{p}-1}}{ \frac{\partial H}{\partial y}(s^{-\mathsf{p}}, s^{\mathsf{q}-\mathsf{p}}\rho(s))}\right)  \leq -1, 
 \end{eqnarray}
where $Ord(\cdot)$ represents the lowest degree of a Laurent series. 

By comparing the coefficients of  terms $\{s^i\}$ in both sides of equation 
\begin{eqnarray}\label{h-part}
	H_{h}^{\ast}(s^\mathsf{p},s^\mathsf{q} \rho(s,h))=0.
\end{eqnarray}
one can easily get 
 $c_0$ is a root of
	$g_{N_i} (1,Y)=0$,
so we can assume $g_{N_i} (1,Y)=(Y- c_0)^{k_{N_i}}g(Y), $ where $g(c_0)\neq 0$.
Letting $c_m(h)$ be the first coefficient depending on $h$ in $\rho(s,h)=s^{\mathsf{q}}(c_0 +\sum_{i=1}^{\infty}c_i(h)s^i)$ and taking the derivative on $h$ in both sides of equation (\ref{h-part}), we have
\begin{eqnarray}\label{c-h}
s^{\mathsf{q}}\left(\sum_{i\geq m} c'_{i}(h)s^i \right) \left( s^{n\mathsf{p}} \frac{\partial H}{\partial y}(s^{-\mathsf{p}},s^{\mathsf{q}-\mathsf{p}}\rho(s))\right)=s^{(n+1)\mathsf{p}},
\end{eqnarray}

 Comparing the coefficients of  terms $\{s^i\}$ on both sides of the above equation, we can get the following estimations:

\begin{itemize}
	\item $0<m\leq 2\mathsf{p}-\mathsf{q}$, i.e., $\mathsf{p}\geq \mathsf{q}/2$, by inequality (\ref{pole});
	\item $N_i+m k_{N_i} \geq (n+1)\mathsf{p}$, 
	this is because  the lowest degree of the left side of the above equation is not more than 
	$\mathsf{q} +m +(N_i -\mathsf{q})+ ( k_{N_i}-1)m= N_i+m k_{N_i}$. 
\end{itemize}
  In addition, from the convexity of the Newton polygon(see Figure \ref{Newton} below), 
  inequalities
  $k_{N_i}\leq n_i $ and $N_i \leq n_i \mathsf{q}$ are both obvious.
Finally,  combining these  inequalities  we have $n_i \geq (n+1)/2.$

	\begin{figure}[htbp]
	\centering
	\includegraphics[scale=1.2]{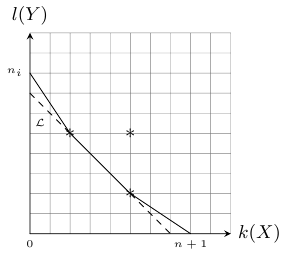}
	\caption{Newton polygon of $H_{h}^{\ast}(X,Y)$}\label{Newton}
\end{figure}
\end{proof}

In general, 	there may  vanishing cycles associated to a singularity in $\Sigma_h$ on $\overline{L_h}$ when $h\rightarrow 0$, which may be non-trivial cycles in $\mathcal{H}_1(\overline{L}_h,\mathbb{Z})$. However, for Hamiltonian systems with an isochronous center, we have the following lemma.
\begin{lemma}\label{non-t-vc}
	For system (\ref{H-C}), if the origin is an isochronous center, then there does not exist a vanishing cycle $\gamma'_h$ associated to a singular point $P^i$ at infinity such that $ \lim_{h\rightarrow 0}\oint_{\gamma'_h}\omega \neq 0$.
\end{lemma}
\begin{proof}

Suppose otherwise, i.e., suppose that such a vanishing cycle $\gamma'_h$ exists and is associated to a point $P^i$ at infinity with the projective coordinate $[1:0:0]$. 

Denote by $\sharp(\overline{L}_h)$ the number(counting multiplicity) of branches determined by the parts of the segments with slope $\leq -1/2$ in the Newton polygons of $H^{\ast}_h$ at $P^i$. 	 On one hand, when  $\gamma'_h\rightarrow 0$  as $h\rightarrow 0$, it  yields  at least one more branch on $\overline{L_0}$ than  $\overline{L_h}$, where the period $1$-form $\omega$ has a pole at $P^i$ with a nonzero residue. So this branch is determined  by a segment with slope $\leq -1/2$ in the Newton polygons of  $H^{\ast}_0$, by Lemma \ref{n/2} and its proof. This implies 
	  $\sharp(\overline{L}_0)>\sharp(\overline{L}_h)$.

On the other hand, 
	 the isochronous center is of Morse type, and the part of $H(x,y)$ with degree $2$ has the form $a_{20}x^2+ a_{11}xy+a_{02}y^2$ such that $a_{20}$ and $a_{11}$ can not be zero simultaneously. So we have 
	 \begin{align}
	 	\begin{array}{rcl}
	 		H^{\ast}_h(X,Y)&=&-hX^{n+1}+a_{20}X^{n-1}+a_{11}X^{n-1}Y +a_{02}X^{n-1}Y^2 \\
	 		&&+\sum_{k=0}^{n-2}\sum_{l= 0}^{k}a_{(n+1-k-l)l}X^{k}Y^{l}.
	 	\end{array}
	 \end{align}
	 
	 Then there are the following three possible cases, and each of them yields a contradiction.

	 \begin{itemize}
  \item If the Newton polygon for $h\neq 0$ does not contain two points $(k_1,l_1)=(n-1,1)$ and $(k_2,l_2)=(n+1,0)$ simultaneously, then  $H^{\ast}_0$ and $H^{\ast}_h$  have the same parts of the segments with slope $\leq -1/2$, which give the same number $\sharp(\overline{L}_0)=\sharp(\overline{L}_h)$.
  \item If the Newton polygon for $h\neq 0$ contains two points $(k_1,l_1)=(n-1,1)$ and $(k_2,l_2)=(n+1,0)$ simultaneously, and the line segment $\mathsf{ls_1}$ with slope $1/2$ contains only two points $(n-1,1)$ and $(n+1,0)$, then 
  then  $H^{\ast}_0$ and $H^{\ast}_h$  have the same shapes of the part of the segments with slope $<- 1/2$, and the branch determined by  $\mathsf{ls_1}$ on $\overline{L_h}$ has a Puiseux parameterization $$X=s,\  Y=s^2\left(h+ \sum_{i> 1} d_i(h) s^i\right), $$ 
   which tends to the branch $Y=0$ on $\overline{L_0}$. Consequently, we still have  $\sharp(\overline{L}_0)=\sharp(\overline{L}_h)$.
  \item If the Newton polygon  for $h\neq 0$ contains two points $(k_1,l_1)=(n-1,1)$ and $(k_2,l_2)=(n+1,0)$ simultaneously, but the line segment $\mathsf{ls_2}$ with slope $-1/2$ contains not only these two points, then we will show that in this case the system (\ref{H-C}) can not be linearizable at the point $(x,y)=(0,0)$.
  
  Let $k_0$ be the maximum value such that $(k_0,l)\in \mathsf{ls_2}$ and $k_0< n-1$. Then in this case  $H(x,y)$ has the form 
  
  \begin{align}
  	\begin{array}{rcl}
  		H(x,y)&=&a_{11}xy +a_{22}y^2 +a_{i_0 i_0}x^{i_0}y^{i_0}+\sum_{i< j,i+j\leq 2i_0}a_{ij}x^{i}y^{j}\\
  &&+\sum_{i\leq j,i+j>2i_0}a_{ij}x^{i}y^{j},\\
  	\end{array}
  \end{align}
  where $ a_{i_0 i_0}\neq 0, \ i_0 =(n+1-k_0)/2$.
 In fact, $i_0 a_{i_0 i_0}$ is nothing other than  the first nonzero linearization constant. This is because, from the results in \cite{Gonsa}, the $(2i_0-2)$-jet of system (\ref{H-C}) has admissible nonlinearities and can be linearized by a transformation of the form 
   \begin{align}
  	\begin{array}{rcl}
  		U&=&x+\sum_{i\leq j}u_{ij}x^{i}y^{j}\\
  V&=&y+\sum_{i< j-1}v_{ij}x^{i}y^{j}
  	\end{array}.
  \end{align}
However, this transformation can not change the resonant terms $i_0 a_{i_0 i_0}x^{i_0}y^{i_0-1} $ in $dx/dt$ and $-i_0 a_{i_0 i_0}x^{i_0-1}y^{i_0} $ in  $dy/dt$.
\end{itemize}	  
\end{proof}

\section{The proof of main theorems}

Now we can prove our main theorems. 

\begin{proof}[Proof of Theorem \ref{h-0-c}]
By Lemma \ref{inf-center}, without loss of generality, we assume $\gamma_h$ has been  divided into $k$ parts $\{\gamma^{j}_h,\ j=1,2,...,k\}$, such that for any point $p\in \gamma_{h}^{j}$, ${\rm i }\varphi(p, t)$ goes to the same point $P^{j}$ at infinity on the same branch when  $t\rightarrow +\infty$. Here $\gamma^{j}_h$ may not be a continuous arc of $\gamma_h$ but  a union of finitely many continuous arcs.
By Lemma \ref{k1}, we can choose a closed orbit $\delta_{h}^{j}$ of $V_h$ encircling and  sufficiently close to $P^{j}$ on this branch.

 We shall prove that, for each $1\leq j\leq k$, there exists a moment $t_j$ such that $ \overline{{\rm i }\varphi(\gamma_{h}^{j}, t_j)}=\delta_{h}^{j}$. 
Then 
 $\gamma_h$ is homologous to the summation of those cycles $\sum_{j=1}^{k}\delta_{h}^{j}$, and the theorem holds since that each $\delta_{h}^{j}$ represents a zero homology cycle in $\mathcal{H}_1(\overline{L_h},\mathbb{Z})$.
 
Suppose otherwise, i.e. suppose that
there exists a number $j$ and a branch of $L_h$ near a point $P^{j} $ such that
the loop $\delta_{h}^{j}$  contains at least two continuous arcs $Ar_{h1}$ and $Ar_{h2}$
satisfying that $Ar_1 \subseteq   {\rm i }\varphi(\gamma_{h}^{j}, t_j)$ 
but $Ar_{h2} \cap {\rm i }\varphi(\gamma_{h}^{j}, t_j)=\emptyset$.

By  the continuation technique,
 we can extend the transformation $\Phi$  to a domain  $\mathcal{D}$ containing  $Ar_{h1}, \ Ar_{h2}$ and  ${\rm i }\varphi(Ar_{h2}, t')$ for any sufficiently small $\left| h\right| $,
  where $t'\in (-t_0, 0]$ and $t_0$ is a real number such that $t_0> t_j$.
  To avoid that ${\rm i }\varphi(Ar_{h2}, t' )$  meets a point at infinity, 
  $Ar_{h2}$ can be shortened properly.
Consequently, for any point $q''\in Ar_{h2}$, letting $q'={\mathrm i }\varphi(q'', -t_j)$, there exists a point $q\in \gamma_h$ such that
 $\Phi(q')=\Phi(q)$ but $q'\not \in \gamma_h$ (see Figure \ref{Figure_3}).
  	\begin{figure}[ht] 	
 	\centering
 	\includegraphics[]{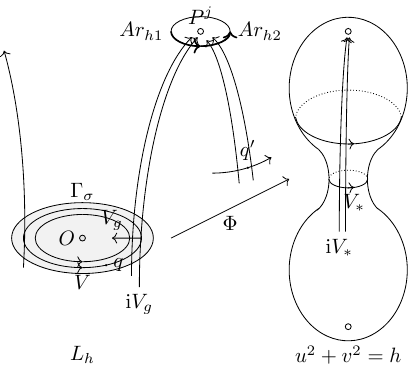}
 	\caption{Proof of Theorem \ref{h-0-c}}
 	\label{Figure_3}
 \end{figure}

Note that  ${\mathrm i }\varphi(Ar_{h2}, -t_j)$  must be a part of a closed orbit $\gamma'_{h}$ of vector field $V$. 
If not so, then for almost every $t$ sufficiently close to $-t_j$, the orbit of $V$ containing ${\mathrm i }\varphi(Ar_{h2}, t)$ tends to a point $P$ at infinity, so  the orbits of ${\rm i}V$ near $P$ either  also tend to $P$ or are closed near $P$, which implies the number $\mathsf{k}\geq 1$ for this branch of $P$. 
However,  $\Phi({\mathrm i }\varphi(Ar_{h2}, t))$ is a part of a closed orbit of $V_{\ast}$, then 
by Lemma \ref{k1} and its proof, the orbit of $V$  must also be closed near $P$, this is a contradiction.

Below we shall show that $\gamma'_h$ must be a vanishing cycle associated to the origin or a singularity at infinity. 

Under the assumption of the theorem, $L_0$ has the same structures at $P^j$  to $L_h$, by Lemma \ref{non-t-vc} and its proof. 
So the above $\mathcal{D}$ can also extended to $L_0$: given two closed orbits $\gamma_0$ and $\delta_0$ of $V_0$ sufficiently close to $O$ and $P^j$ on the  branch of $L_0$ that is the limit of branch containing $Ar_{h2}$ on $L_h$ when $h\rightarrow 0$, then $ \overline{{\rm i }\varphi(\gamma_{0}, t)}\not =\delta_{0}$ for any $t$.
Thus $\mathcal{D}$ can contain
  a trajectory $ Ar_{02}\subset \delta_0 $ but $ \not \subset \overline{{\rm i }\varphi(\gamma_{0}, t)}$ and  
    ${\rm i }\varphi(Ar_{02}, t')$  for some a $t'\in (-\infty, 0]$. 
    
  Noticing that  $L_0$ has only a single finite singular point, 
  ${\rm i }\varphi(Ar_{02}, t')$ can be well defined for $t'\rightarrow -\infty$ and its limit is either the origin or a point at infinity, by Lemma \ref{inf-center} and its proof. Besides,  vector field $V_g$ coincides with $-{\rm i}V$ on $L_0$, so
 ${\mathrm i }\varphi(Ar_{h2}, -t_j)$ will tend to  
$\lim_{t'\rightarrow-\infty}{\rm i }\varphi(Ar_{02}, t')$ along the trajectory of $V_g$. This means $\gamma'_h$ is a vanishing cycle of a singularity at infinity or the origin. 

However, by Lemma \ref{non-t-vc} and the assumption on the period $1$-form, the former case is impossible. As for the latter case,  
recalling   that $\Phi(q')=\Phi(q)$ but ${\rm i }\varphi(Ar_{h2}, -t_j)\bigcap \gamma_h =\emptyset$ and $\Phi$ is a homeomorphism near the origin,
it is also impossible.

\end{proof}

\begin{proof}[Proof of Corollary \ref{J-C}]
	If the linearization change $\Phi $ is 
	well defined on the whole plane $\mathbb{C}^2$, for instance, 
	polynomial map $\Phi_P$ appearing in the Jacobian conjecture, then it maps a  small disc punctured by a  pole of the period $1$-form $\omega $ to a  small disc(topologically) punctured by a  pole of $1$-form $-d u/v$ on $L_h$ for any $h$. This means that $\omega $ dose not have poles with zero residue at infinity. 
	
	Besides, for polynomial map $\Phi_P$, due to 
$$\det \left( \begin{array}{cc}
	\frac{\partial f}{\partial x}	& \frac{\partial f}{\partial y} \\
	\frac{\partial g}{\partial x}	& \frac{\partial g}{\partial y}
\end{array}\right) \not =0,$$
the singularities on critical level curve $H(x,y)=(f^2 +g^2)/2=0$
are just intersections of two algebraic curves $f=0$ and $g=0$.
Thus, by Theorem \ref{h-0-c} and Proposition 6.1 in \cite{Gavri}, the corollary holds.
\end{proof}
	
\begin{proof}[Proof of Theorem \ref{non-iso}]

Noticing that  $\omega$ has a pole at infinity is equivalent to say $\mathsf{k} \geq 1$ for system (\ref{vf-inf}),
this theorem is a direct conclusion of Lemma \ref{inf-center}
 and Lemma \ref{n/2}.
\end{proof}

\section{Non-isochronicity of real  Hamiltonian systems of even degree $n$}

In the last section we focus on the relation between 
the Gavrilov's question and Jarque-Villadelprat conjecture.
It is worthy mentioned that the latter  is not true in the complex setting, some counterexamples can be found in Gavrilov's paper \cite{Gavri}. 
Firstly we shall prove Theorem \ref{G-J-V}.

\begin{proof}[Proof of Theorem \ref{G-J-V}]

If $H(x,y)$ is a real polynomial of odd degree $n+1$,
then the real algebraic curve $L_h$ has at least two connected components on  $\mathbb{R}^2$,
one of them is just the closed orbit $\gamma_h$ near the center which can represent  the corresponding vanishing cycle,
and another one, denoted by $\gamma'_h$,  tends to a point at infinity.
The real systems can be embedded   in $\mathbb{C}^2=\{(x,y)=(x_1+{\rm i}x_2, y_1+{\rm i}y_2)\}\cong \{(x_1,x_2,y_1,y_2)\}= \mathbb{R}^4$.
Then the real plane $\mathbb{R}^2$ is a subset defined by $x_2=y_2=0$,
and the closed orbit $\gamma_h$ on $\mathbb{R}^2$ can be represented by $H(x_1 ,y_1)=h$.

If $$\Phi_{\mathbb{R}}: \mathbb{R}^2 \rightarrow \mathbb{R}^2,\ \  (x_1,y_1)\mapsto (u_1,v_1)=(\phi_1(x_1,y_1),\phi_2(x_1,y_1))$$ is the transformation linearizing real isochronous center,
then  the following map
$$\Phi: \mathbb{C}^2 \rightarrow \mathbb{C}^2,\ \ (x,y)\mapsto (u,v)=(\phi_1(x,y),\phi_2(x,y))$$  can  linearized the complex isochronous center of system (\ref{H-C}).
Denoting by $$\Pi: \mathbb{C}^2 \rightarrow \mathbb{C}^2,\ \ (x,y)\mapsto (\overline{x},\overline{y})$$ 
the conjugate operation on $\mathbb{C}^2$,
we have $\Phi \circ \Pi=\Pi \circ \Phi$, since 
$$(\overline{u},\overline{v})=(\overline{\phi_1(x,y)},\overline{\phi_2(x,y)})=(\phi_1(\overline{x},\overline{y}),\phi_2(\overline{x},\overline{y})).$$

If $\gamma_h$ is a trivial cycle on the closure $\overline{L_h}$ of the generic complex curve $L_h$ defined by $H(x,y)=h$,
then $L_h$ is divided into two path-connected open components $ A_1$ and $A_2$ such that $A_1\cap A_2=\emptyset$ and their common boundary is $\gamma_h$.
Without loss of generality, we assume $\gamma'_h\cap A_1\neq \emptyset$ 
and can construct a smooth curve $l_h\subset L_h$ connecting two points $p\in \gamma_h$ and $p'\in \gamma'_h\cap A_1$,
such that  $l_h$ intersects $\gamma_h$ transversally at only one point $p$.

In the domain $\mathcal{D}$ where $\Phi$ is well defined,
we have $\Pi(A_1\cap \mathcal{D})\subset A_2$, because  $\Phi$ is a homeomorphism so that $\Phi(A_1\cap \mathcal{D})\cap\Phi( A_2\cap \mathcal{D})=\emptyset$
and $\Pi(\Phi(A_1\cap \mathcal{D}))\subset \Phi( A_2\cap \mathcal{D})$.
Therefore  the   complex conjugate $\overline{l_h}$ of $l_h$ belongs to $A_2$ in  $\mathcal{D}$. Finally $l_h\cup \overline{l_h}\cup \{p,p'\}$ forms a closed curve  intersecting $\gamma_h$ at only one point $p$ with intersection number $1$ on $L_h$(see Figure \ref{Figure_4}).
This means that $\gamma_h$ can not be trivial on   $\overline{L_h}$, which leads a contradiction.
	\begin{figure}[ht]
	
	\centering
	\includegraphics[scale=]{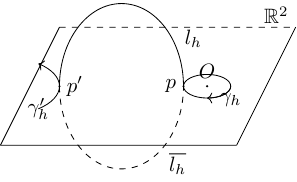}
	\caption{$l_h$ and its  complex conjugate}
	\label{Figure_4}
\end{figure}	
\end{proof}

By Theorem \ref{G-J-V}, we observe an interesting relation between Gavrilov's question
and Jarque-Villadelprat conjecture, that is,
if the latter conjecture is not true, then such real systems possessing isochronous centers provide a negative answer to Gavrilov's question.

In the end, as applications  of Theorem \ref{h-0-c}, 
we present a conclusion which verifies the Jarque-Villadelprat conjecture for a large class of real systems.
Note that in the real setting, an isochronous center must be a non-degenerated singularity, i.e., it must be of Morse type,
so we have

\begin{corollary}
For a real polynomial Hamiltonian system (\ref{H-C}) of even degree, 
if each (complex) critical level curve having a center contains only a single singularity, and the period $1$-form has no pole at infinity with zero residue on any level curve, then it does not admit any isochronous center at all.
\end{corollary}

\section*{Acknowledgements}
This work is supported by NSFC 11701217 and
NSF 2017A030310181 of Guangdong  Province(China).


\begin{thebibliography}{00}
	
	


\bibitem{AGR} B. Arcet, J. Gine and V.G. Romanovski, Linearizability of planar polynomial Hamiltonian systems,
Nonlinear Analysis: Real World Applications 63 (2022), 103422, 19.

\bibitem{Gonsa} J. Basto-Goncalves, Linearization of resonant vector fields, Trans. Amer. Math. Soc. 362 (12) (2010), 6457-6476. 


\bibitem{Cama-LN-Sad-1} C. Camacho, A. Lins Neto, and P. Sad,
Topological invariants and equidesingularization for holomorphic systems, J. Differential Geom. 20 (1984), no. 1, 143-174.
	
	
	
	
	
	
	\bibitem{C-M-V} A. Cima, F. Ma\~{n}osas,   J. Villadelprat, Isochronicity for several classes of Hamiltonian systems,
	J. Differential Equations 157, 373-413 (1999).

\bibitem{Cre-Pala} J. Cresson, J. Palafox, Isochronous centers of polynomial
Hamiltonian systems and a conjecture of Jarque and Villadelprat,
J.  Differential Equations 266, 5713-5747 (2019).
	
	
	
\bibitem{Fischer} G. Fischer,
Plane algebraic curves,
Translated from the 1994 German original by Leslie Kay.
Student Mathematical Library 15. American Mathematical Society, Providence, RI, 2001.
	
	
	\bibitem{Gavri}  L. Gavrilov, Isochronicity of plane polynomial Hamiltonian systems, Nonlinearity 10 (1997), 433-448.
	
	
\bibitem{Ilya-Yako}Y.	Ilyashenko, S. Yakovenko,
	Lectures on analytic differential equations,
	Graduate Studies in Mathematics 86, American Mathematical Society, Providence, RI, 2008.
	
	\bibitem{J-V}  X. Jarque and J. Villadelprat,   Nonexistence of isochronous centers in planar polynomial Hamiltonian systems of degree four.
	J. Differential Equations 180, 334-373 (2002).


	
	\bibitem{Kir} F. Kirwan, Complex Algebraic Curves,  London Mathematical Society, Student Text 23,
	Cambridge University Press, Cambridge, 1992.
	
	\bibitem{L-R}  Llibre, J.,  Romanovski, V.G.: Isochronicity and linearizability of planar polynomial Hamiltonian systems.
	J. Differential Equations 259, 1649-1662 (2015).
	
	
	
	\bibitem{M-V}  F. Ma\~{n}osas and J. Villadelprat,  Area-preserving normalizations for centers of planar Hamiltonian 
	J. Differential Equations 179, 625-646 (2002).
	
	
	
	
\end{thebibliography}
\end{document}